\documentclass[11pt]{amsart}
\usepackage{amsmath,amsthm,amssymb,amsfonts, esint}
\usepackage[nobysame,abbrev,alphabetic]{amsrefs}
\usepackage{url}
\usepackage{ esint }
\usepackage{graphicx}
\usepackage{color}
\usepackage{times}
\usepackage{cite}
\usepackage[bottom]{footmisc}
\usepackage{latexsym}
\usepackage{graphicx}
\usepackage{verbatim}
\usepackage[top=2cm, bottom=2.8cm, left=3cm, right=3cm]{geometry}%

\newtheorem{thm}{Theorem}[section]
\newtheorem*{theorem*}{Theorem}
\newtheorem*{acknowledgement*}{Acknowledgement}
\newtheorem{corollary}[thm]{Corollary}

\newtheorem{lemma}[thm]{Lemma}
\newtheorem{proposition}[thm]{Proposition}
\theoremstyle{definition}

\theoremstyle{remark}
\newtheorem{remark}[thm]{Remark}
\newtheorem{conjecture}[thm]{Conjecture}

\numberwithin{equation}{section}

\newcommand\de{{\partial}} 
\newcommand{\abs}[1]{\left\vert#1\right\vert}

\newcommand{\R}{\mathbb{R}}

\DeclareMathOperator{\Vol}{Vol}
\DeclareMathOperator{\Diam}{Diam}
\DeclareMathOperator{\Area}{Area}
\DeclareMathOperator{\MinA}{MinA}

\def\XXint#1#2#3{{\setbox0=\hbox{$#1{#2#3}{\int}$ }
\vcenter{\hbox{$#2#3$ }}\kern-.6\wd0}}

\title{Warped Tori with Almost Non-negative Scalar Curvature}

\author{Brian Allen}
\address[Brian Allen]{Lehman College, CUNY, 250 Bedford Park Blvd W, Bronx, NY 10468}
\email{Brian.Allen1@lehman.cuny.edu}

\author{Lisandra Hernandez-Vazquez}
\address[Lisandra Hernandez-Vazquez]{Stony Brook University, 100 Nicolls Road, Stony Brook, NY 11794, USA.}
\email{lisandra.hernandezvazquez@stonybrook.edu}

\author{Davide Parise}
\address[Davide Parise]{University of California San Diego, 9500 Gilman Drive \# 0112, La Jolla, CA  92093-0112, USA.}
\email{dado.parise@gmail.com}

\author{Alec Payne}
\address[Alec Payne]{Duke University, 120 Science Drive, Room 029B, Durham, NC 27701.}
\email{alec.payne@duke.edu}

\author{Shengwen Wang}
\address[Shengwen Wang]{Queen Mary University of London, Mile End Road, E1 4NS, London, UK.}
\email{shengwen.wang@qmul.ac.uk}

\begin{document}
\date{\today}
\begin{abstract}
For sequences of warped product metrics on a $3$-torus satisfying the scalar curvature bound $R_j \geq -\frac{1}{j}$, uniform upper volume and diameter bounds, and a uniform lower area bound on the smallest minimal surface, we find a subsequence which converges in both the Gromov-Hausdorff (GH) and the Sormani-Wenger Intrinsic Flat (SWIF) sense to a flat $3$-torus. 
\end{abstract}
\maketitle
\section*{Introduction} 

The Scalar Torus Rigidity Theorem states that any Riemannian manifold which is diffeomorphic to an $n$-dimensional torus and which has nonnegative scalar curvature is isometric to a flat torus. 
It is called a rigidity theorem because it is a theorem which forces a
Riemannian manifold to have a rigid structure: in this case to be isometric to a flat torus.
This theorem was proven for dimension $n=3$ by Schoen and Yau in 1979 \cite{Schoen-Yau-min-surf}, using results from minimal surface theory that can now be extended to higher dimensions.   Gromov and Lawson gave a proof in all dimensions using the Lichnerowicz formula in \cite{Gromov-Lawson-1980}.  

Recently, Gromov suggested that sequences of manifolds diffeomorphic to tori with almost non-negative scalar curvature and appropriate compactness conditions should converge to flat tori  \cite{Gromov-Dirac}. By work of Gromov \cite{Gromov-Dirac} and of Bamler \cite{Bamler-16}, if one assumes additional conditions on the metric tensors to guarantee that they converge in the $C^0$ sense then one can obtain $C^0$ convergence of this sequence of tori with almost non-negative scalar curvature to flat tori. Since there are known examples of sequences without these additional hypotheses which do not converge in the $C^0$ or even Gromov-Hausdorff (GH) sense it was suggested by Gromov that the conjecture should be in terms of Sormani-Wenger Intrinsic Flat (SWIF) convergence. In \cite{Sormani-scalar}, Sormani formulated a precise conjecture for such a sequence of tori with almost non-negative scalar curvature as follows. 
\begin{conjecture}\label{MainConjecture}
Let $M_j = (\mathbb{T}^3, g_j)$ be a sequence of Riemannian manifolds diffeomorphic to a $3$-torus such that 
\begin{align} \label{HypothesisConjecture}
R_{j} \ge -\frac{1}{j}, \hspace{0.5cm} \Vol(M_j) \le V_0, \hspace{0.5cm} \Diam(M_j) \le D_0 \,\,\, \text{and} \,\,\, \MinA(M_j) \ge A_0 > 0,
\end{align}
where $R_j$ is the scalar curvature and $\MinA(M_j)$ is the area of the smallest closed minimal surface in $M_j$. 
Then, there is a subsequence of $M_j$ converging in the SWIF sense to a flat torus: $M_{j_k} \stackrel {SWIF}{\longrightarrow} M_{\infty}$, where $M_\infty$ is a flat torus. 
\end{conjecture}
 Note that if any of the assumed conditions on the sequence in this conjecture are relaxed, then there are known counterexamples. The uniform volume and diameter bounds are necessary to prevent expansion and collapsing. The $\MinA$ condition is necessary to prevent bubbling and \textquotedblleft sewing" examples which would otherwise provide counterexamples to this conjecture \cite{Basilio-Dodziuk-Sormani}. The $\MinA$ condition is natural in this setting given the crucial role that stable minimal surfaces played in Schoen-Yau's proof of the torus rigidity theorem \cite{Schoen-Yau-min-surf}. Also, the $\MinA$ condition has appeared in the rigidity results of Bray, Brendle and Neves \cite{BBN-CAG-10} for area minimizing 2-spheres in 3-spheres and Bray, Brendle, Eichmair and Neves \cite{BBEN-CPAM-10} for area minimizing projective planes in 3-manifolds.
 
 Moreover, there are counterexamples to Conjecture \ref{MainConjecture} if SWIF convergence is replaced with GH convergence. Basilio and Sormani constructed sequences of tori satisfying the hypotheses of this conjecture with no GH limit and a GH limit to a non-smooth space that is not the flat torus \cite{Basilio-Sormani-1}. These examples have increasingly thin wells with positive scalar
curvature surrounded by an annular region with $R_{j} \ge -\frac{1}{j}$.  Since thin wells disappear
under SWIF convergence, these examples converge in the SWIF sense.
On the other hand, all of their examples converge in SWIF sense to a flat torus.

The first paper to apply SWIF convergence in the setting of positive scalar curvature was the paper by Lee and Sormani \cite{LeeSormani1} where sequences of rotationally symmetric, asymptotically flat manifolds with ADM mass tending to zero are shown to converge to regions in Euclidean space under SWIF convergence. In this case there are counterexamples given by Lee and Sormani where sequences with the properties above do not converge under GH convergence and hence SWIF convergence is essential. This informs the intuition that SWIF convergence is well suited for convergence questions where positive scalar curvature is natural. This intuition inspires the use of SWIF convergence in Conjecture \ref{MainConjecture} and is reinforced by the results of this paper.

 In this paper, we will prove Conjecture \ref{MainConjecture} in the setting where the metrics are assumed to be warped product metrics.   This setting was first suggested by Sormani after formulating Conjecture \ref{MainConjecture} \cite{Sormani-scalar}.   We find a subsequence which converges in both SWIF and GH sense and we note that a subsequence is necessary because the
sequence could have subsequences converging to different flat tori.  It is perhaps surprising that we
obtain GH convergence as this means that our sequences are not developing long thin wells as in the examples in \cite{Basilio-Sormani-1}.   

In particular, we are going to consider the following two special cases: \\
\noindent (i) \textit{Doubly Warped Products:} For $x, y, z \in [- \pi, \pi]$ and positive $a_j, b_j \colon [- \pi, \pi] \rightarrow \mathbb{R}$, the metric
\begin{equation}
g_j = a_j^2(z) dx^2 + b_j^2(z) dy^2 + dz^2,
\end{equation}
is a doubly warped product. \\
\noindent (ii) \textit{Singly Warped Products:} For $x, y, z \in [- \pi, \pi]$ and positive $f_j \colon [- \pi, \pi] \times [-\pi, \pi] \rightarrow \mathbb{R}$, the metric
\begin{equation}
g_j = dx^2 + dy^2 + f_j^2(x, y) dz^2,
\end{equation}
is a singly warped product. \\

Throughout the rest of this paper, by \textquotedblleft doubly warped product" we will be referring to item (i) above and by \textquotedblleft singly warped product" we will be referring to item (ii) above. Now, we state our main result for doubly warped products. 

The main result for doubly warped products:
\begin{thm}\label{MainThmCase1}
Suppose we have a sequence $M_j = (\mathbb{T}^3, g_j)$, where each $g_j$ is a doubly warped product satisfying
\begin{equation}
R_{j} \ge -\frac{1}{j}, 
\hspace{0.5cm} \Diam(M_j) \le D_0, \,\,\,\text{and} \,\,\, \MinA(M_j) \ge A_0 > 0,
\end{equation}
then there exists a subsequence $M_{j_k}$ converging uniformly to a flat torus. In particular, $M_{j_k}$ converges in the GH and SWIF sense to a flat torus. 
\end{thm}

To prove Theorem \ref{MainThmCase1}, we first show in Theorem \ref{W12convergence} that the scalar curvature bound allows us to find subsequences of the warping functions that converge to nonzero constants in $W^{1,2}(S^1)$. A key step in obtaining these convergent subsequences is the existence of upper and lower uniform bounds on the warping functions found in Proposition \ref{UpperLowerBoundsDoubly}. We show these bounds can be derived from the MinA and diameter bounds in the hypotheses of our theorem. 
It then follows from Morrey's inequality for one dimensional domains that in fact we have $C^{0, \frac{1}{2}}$ convergence. From here we obtain uniform, GH, and SWIF convergence. Note that we did not use a uniform volume bound, yet this is necessary for Conjecture \ref{MainConjecture} to hold in general.

The main result for singly warped products:

\begin{thm}\label{MainThmCase2}
Suppose we have a sequence $M_j = (\mathbb{T}^3, g_j)$, where $g_j$ is a singly warped product satisfying 
\begin{equation}
R_{j} \ge -\frac{1}{j}, 
\hspace{0.5cm} \Vol(M_j) \le V_0, 
\,\,\,\text{and} \,\,\, \MinA(M_j) \ge A_0 > 0,
\end{equation}
Then, there exists a subsequence $M_{j_k}$ converging uniformly to a flat torus. In particular, $M_{j_k}$ converges in the GH and SWIF sense to a flat torus.
\end{thm}

To prove Theorem \ref{MainThmCase2}, we find in Lemma \ref{f_j convergence} a subsequence of the warping functions $f_j$ that converges to a positive constant in $W^{1,1}(\mathbb{T}^2)$. Completely different techniques than those of Theorem \ref{MainThmCase1} are used. The proof involves using the Moser-Trudinger inequality in Proposition \ref{LinfinityEsthj}  to gain an $L^p$ bound, $p>2$, on $f_j$, combined with control obtained from the $\MinA$ lower bound in Lemma \ref{bound from minA singly warped}. Then, we use a maximum principle on a certain operator to obtain $C^0$ control from below on the warping functions in Corollary \ref{C0Boundh}, which then allows us to appeal to a result of the first named author to find that a subsequence has the desired convergence to a flat torus~\cite{BALp}. Note that we do not use a uniform diameter bound.

We now give a brief outline of the paper:
In Section \ref{sec: background} we describe the definitions and previous theorems which will be essential to understanding the results of this paper. In the interest of keeping the background concise we offer up references to interesting definitions and results which are not essential to understanding the main results of this paper. In Section \ref{sec: Case 1} the proof of Theorem \ref{MainThmCase1} is given and in Section \ref{sec: Case 2} the proof of Theorem \ref{MainThmCase2} is given. In both sections many interesting estimates are developed which give potential insight into the full conjecture \ref{MainConjecture}. 

\begin{acknowledgement*}
We would like to thank Kai Xu for pointing out a mistake in section \ref{subsec: Convergence of f_j Case 2} which has now been corrected in this version of the paper with a different, simpler argument. The statement of the main theorems have not changed.

This research began at the Summer School for Geometric Analysis, located at and supported by the Fields Institute for Research in Mathematical Sciences. We would like to thank the organizers of this summer school, Spyros Alexakis, Walter Craig, Robert Haslhofer, Spiro Karigiannis and McKenzie Wang.  Finally, the authors would like to thank Christina Sormani for her direction, advice, and constant support. Specifically, Christina brought this team together during the Fields institute to work on this project, organized workshops at the CUNY graduate center which the authors participated in, and provided travel support to the team members. Christina Sormani is supported by NSF grant DMS-1612049.
Brian Allen is supported by the USMA. Davide Parise is partially supported by the Swiss National Foundation grant 200021L\_175985.
\end{acknowledgement*}

\section{Background} \label{sec: background}
\noindent In this section, we review some basic definitions and facts that will be used throughout the paper.

We start by reviewing the notion of uniform convergence of metric spaces. Consider two metric spaces $(X,d_1)$, $(X,d_2)$ and define the uniform distance between these metric spaces to be
\begin{align}
d_{unif}(d_1,d_2) = \sup_{x,y\in X} |d_1(x,y) - d_2(x,y)|.
\end{align}
Notice that if you think of the metrics as functions, $d_i: X\times X \rightarrow \R$, then the uniform distance $d_{unif}(d_1,d_2)$ is equivalent to the $C^0$ distance between functions. We say that a sequence of metrics spaces $(X, d_j)$ converges to the metric space $(X,d_{\infty})$ if $d_{unif}(d_j, d_{\infty}) \rightarrow 0$ as $j \rightarrow \infty$.

One limitation of uniform convergence is that it requires the metric spaces to have the same topology and so other important notions of convergence have been introduced which do not depend on topology. Two particularly important notions of convergence for metric spaces and Riemannian manifolds are Gromov-Hausdorf (GH) convergence and Sormani-Wenger Intrinsic Flat convergence (SWIF). In this paper we will be able to show GH and SWIF convergence but due to the symmetries of the metrics assumed we will also be able to show uniform convergence and so we will not define these notions in this paper. For the definition of GH convergence see \cite{BBI} and the references therein. For the definition of SWIF convergence see \cite{SW-JDG}.

In the case of doubly warped products we will be able to show $C^{0,\frac{1}{2}}$ convergence of the warping functions $a_j(z), b_j(z)$ to constants in section \ref{sec: Case 1}. We will then wrap up the proof of Theorem \ref{MainThmCase1} by applying the following corollary of Proposition 3.7 in \cite{Gromov-metric2}, for the case of GH convergence, and a corollary of Theorem 5.6 in \cite{SW-JDG}, for the case of SWIF convergence.

\begin{corollary}\label{SW-SmoothConvToSWIF}
If a sequence of Riemannian manifolds $M_j=(M, g_j)$ converges to the Riemannian manifold $M_{\infty}=(M, g_{\infty})$ in the $C^{0,\alpha}$ sense then $M_j$ converges in GH and SWIF to $M_{\infty}$ as well. 
\end{corollary}

It is important to note that showing $C^{0,\alpha}$ convergence of the warping functions is equivalent to showing $C^{0,\alpha}$ convergence of the Riemannian manifolds in the doubly warped product case. 

In the singly warped product case we will not be able to show $C^{0,\frac{1}{2}}$ convergence of the warping functions but instead will be able to show $L^p$ convergence for $p>2$ and $C^0$ convergence from below. By combining these estimates with Theorem 1.4 of \cite{BALp} we will be able to conclude uniform, GH, and SWIF convergence. We now move on to produce the estimates needed to apply Corollary \ref{SW-SmoothConvToSWIF} and Theorem Theorem 1.4 of \cite{BALp} in order to prove Theorem \ref{MainThmCase1} in section \ref{sec: Case 1} and prove Theorem \ref{MainThmCase2} in section \ref{sec: Case 2}.

\section{Doubly Warped Products of One Variable} \label{sec: Case 1}

In this section, we will prove Theorem \ref{MainThmCase1}. Recall that we are considering a sequence of doubly warped product metrics $g_j$ on $\mathbb{T}^3$ such that $x, y, z \in [- \pi, \pi]$ and $a_j, b_j \colon [- \pi, \pi] \rightarrow \mathbb{R}$ positive functions, and
\begin{equation}\label{Case1}
g_j = a_j^2(z) dx^2 + b_j^2(z) dy^2 + dz^2.
\end{equation}

\subsection{Scalar Curvature of Doubly Warped Products} \label{subsec: ScalarEq Case 1}
In order to prove Theorem \ref{MainThmCase1} we will need to find an expression for the scalar curvature of a doubly warped product. The resulting differential inequality from $R_j \geq -\frac{1}{j}$ will be key to showing the desired convergence.

\begin{lemma}
The scalar curvature for a metric $g= a(z)^2dx^2 + b(z)^2dy^2 + dz^2$ on $\mathbb{T}^3$ is 
\begin{equation} \label{scalarcurvCase1}
R = -2 \left(\frac{a''}{a}+\frac{b''}{b} + \frac{a' b'}{ab}\right).
\end{equation}
\end{lemma}
\begin{proof} By Section 4.2.4 of Petersen's book \cite{Petersen}, a metric of this form has the following Ricci curvature.
\begin{align}
\text{Ric}\left(\frac{\partial}{\partial x}\right) &= \left(-\frac{a''}{a} - \frac{a' b'}{ab}\right)\frac{\partial}{\partial x}
\\\text{Ric}\left(\frac{\partial}{\partial y}\right) &= \left(-\frac{b''}{b} - \frac{a' b'}{ab}\right)\frac{\partial}{\partial y}
\\\text{Ric}\left(\frac{\partial}{\partial z}\right) &= \left(-\frac{a''}{a} - \frac{b''}{b}\right)\frac{\partial}{\partial z}
\end{align}
Thus, we have the conclusion of this lemma.
\end{proof}

This lemma means that under the conditions of Theorem \ref{MainThmCase1}, the condition $R_j \geq -\frac{1}{j}$ translates into the following condition on $a_j$ and $b_j$
\begin{equation}
\frac{a''_j}{a_j}+ \frac{b''_j}{b_j} + \frac{a'_j b'_j}{a_j b_j} \leq \frac{1}{2j}.
\end{equation}

\subsection{Diameter Bounds, the MinA Condition and Uniform Bounds} \label{subsec: MinA Case 1}
We will now investigate the consequences of the MinA hypothesis, with a particular emphasis on how this translates into natural lower and upper bounds for the warping functions. 
We start with the so-called $\MinA$ condition, according to which the smallest possible area of a closed minimal surface in $M_j$ is bounded from below by a certain constant: 
\begin{equation} \label{EquationMinA}
\MinA (M_j) = \inf \{ \Area(\Sigma) | \; \Sigma \text{ is a closed minimal surface in $M_j$}\} \geq A_0 > 0.
\end{equation}
Notice that this lower bound is uniform in $j$. 
Intuitively speaking, this condition allows us to control better the geometry of the $M_j$'s, for instance by avoiding bubbling phenomena and sewing counterexamples. What's more, a careful analysis of these pathological construction yields that the $\MinA$ hypothesis is not only a simplification of the problem but also a rather natural notion. 
For further details on those examples where $\MinA(M_j) \rightarrow 0$ we refer to \cite{Basilio-Sormani-1}. A notion related to the $\MinA$ hypothesis has been used by Bray, Brendle and Neves, in \cite{BBN-CAG-10}, to prove a cover splitting rigidity theorem and by the same authors with Eichmair, in \cite{BBEN-CPAM-10}, to prove a rigidity theorem concerning $\mathbb{RP}^3$.

Our first result is that (2.1) yields a pointwise lower bound, independent of $j$, on the product $a_j(z)b_j(z)$ and uniform lower bounds on the integrals of $a_j$ and $b_j$.

\begin{lemma}\label{subsec: bound from minA doubly warped}
Let $M_j = (\mathbb{T}^3, g_j)$ as in (\ref{Case1}).  If $\MinA(M_j) \geq A_0$, then for all $z \in [-\pi, \pi]$,
\begin{align}
a_j(z)b_j(z) &\geq \frac{A_0}{4\pi^2} \\
\int_{-\pi}^\pi a_j(z) dz &\geq \frac{A_0}{2\pi}, \\
\int_{-\pi}^\pi b_j(z) dz &\geq \frac{A_0}{2\pi}.
\end{align}
\end{lemma}
\begin{proof}
Consider the three homotopy classes 
\begin{equation}
[x,y,0): x,y \in S^1], [(x,0,z): x,z \in S^1] \; \text{and} \; [(0,y,z): z,y \in S^1],
\end{equation} 
in the three dimensional torus $\mathbb{T}^3$. These are just the homotopy classes of two dimensional tori in our manifold. 
By a result of Schoen-Yau \cite{Schoen-Yau-min-surf}, we can find a minimal surface in each of these homotopy classes. So, if $\phi_{z=0}(x,y): \mathbb{T}^2 \to M_j$ is the embedding of the representative $(x,y,0)$ into our manifold $M_j$, its area satisfies
\begin{equation}
Area(\phi_{z=0}(x,y))\geq \MinA(M_j)\geq A_0 > 0
\end{equation}
Similarly,
\begin{equation}
Area(\phi_{x=0}(z,y)) \geq A_0>0,
\end{equation}
\begin{equation}
Area(\phi_{y=0}(x,y)) \geq A_0>0.
\end{equation}
Let $\omega$ be the 2-form  
$ a_j(z)b_j(z)dx\wedge dy$ obtained by contracting the volume form with $\frac{\de}{\de z}$. Then 
\begin{align}
\Area(\phi_{z=0}(x,y)) &= \int_{-\pi}^\pi \int_{-\pi}^\pi \phi_{z=0}^*(\omega)
= \int_{-\pi}^\pi \int_{-\pi}^\pi a_j(0)b_j(0) dxdy
= 4\pi^2 a_j(0)b_j(0).
\end{align}
Observe that we could have chosen any other $z$-level set. For any $z_0$, 
\begin{equation}
a_j(z_0)b_j(z_0) \geq \frac{\Area(\phi_{z=z_0}(x,y))}{4\pi^2} \geq \frac{A_0}{4\pi^2}
\end{equation}
This establishes the first part of the theorem. 
 
For the other two parts of the theorem, we just compute the areas of the embeddings $\phi_{x=0}$ and $\phi_{y=0}$ and apply the same argument as above. 
The computations here give
\begin{align}
\Area(\phi_{x=0}(y,z)) = \int_{-\pi}^\pi \int_{-\pi}^\pi b_j(z) dydz =2\pi \int_{-\pi}^\pi b_j(z) dz
\end{align}
and 
\begin{align}
\Area(\phi_{y=0}(x,z)) = \int_{-\pi}^\pi \int_{-\pi}^\pi a_j(z) dxdz =2\pi \int_{-\pi}^\pi a_j(z) dz.
\end{align}
Therefore we can find constants $C_1, C_2$ giving the last two estimates in the theorem.
\end{proof}

We now investigate the diameter bound $\Diam(M_j)\leq D_0$ and find uniform upper and lower bounds for $a_j$ and $b_j$. In doing so, we need the following two lemmas regarding the warping functions $a(z), b(z)$ on a fixed $M_j$.

\begin{lemma}\label{lower}
Let $M_j = (\mathbb{T}^3, g_j)$ as in (\ref{Case1}). If $M_j$ has diameter $Diam(M_j) \leq D_0$, then
\begin{equation} \label{productbound}
\min_{z \in [-\pi,\pi]}a_j(z)\leq D_0 \quad \text{and} \quad \min_{z \in [-\pi,\pi]}b_j(z)\leq D_0,
\end{equation} 
\end{lemma}
\begin{proof}
Consider two points on the torus $P_1 = (0,0,0)$ and $P_2 = (1, 0, 0)$. For $t \in [0,1]$, let $\Gamma(t)=(x(t),y(t),z(t))$ be the minimal geodesic with $\Gamma(0) = P_1$ and $\Gamma(1) = P_2$. We may think of $\Gamma(t)$ as a path in $\mathbb{R}^3$ starting at $(0,0,0)$ and ending at 
$(1+2\pi n_1, 2\pi n_2, 2\pi n_3)$ for some $n_1, n_2, n_3 \in \mathbb{Z}$. Thus,
\begin{equation}
1 \leq \abs{\int_{0}^1 x'(t) dt}.
\end{equation}
Now, let $z_1$ be such that $a_j(z_1) = \min_{z \in [-\pi, \pi]}a_j(z)$, which is positive by assumption (\ref{Case1}). Note that $z_1$ depends on $j$.  Then,
\begin{align}
\min_{z \in [-\pi, \pi]}a_j(z) = a_j(z_1) &\leq a_j(z_1) \abs{\int_{0}^{1}x'(t) dt} 
\\&\leq \int_{0}^{1} a_j(z_1) \abs{x'(t)} dt
 \\&\leq \int_{0}^{1} \sqrt{a_j(z_1)^2 x'(t)^2 + b_j(z(t))^2y'(t)^2 + z'(t)^2} dt
\\&\leq \int_{0}^1 \sqrt{a_j(z(t))^2 x'(t)^2 + b_j(z(t))^2y'(t)^2 + z'(t)^2} dt 
\\&= \mathrm{Length}(\Gamma) \leq \Diam(M_j) \leq D_0
\end{align}
We may do the same for $b_j$ using a minimal geodesic connecting $(0,0,0)$ and $(0,1,0)$. 
\end{proof}
\begin{lemma}\label{logl2bound}
Let $M_j = (\mathbb{T}^3, g_j)$ as in (\ref{Case1}). If $R_j \geq -\frac{1}{j}$, then the functions $\alpha_j(z):=\ln(a_j(z))$ and $\beta_j(z):=\ln(b_j(z))$ satisfy 
\begin{equation}\label{gradbound}
\int_{-\pi}^{\pi} {\alpha'_j}^2 dz \leq \frac{2\pi}{j} \quad \text{and} \quad \int_{-\pi}^{\pi} {\beta'_j}^2 dz\leq \frac{2\pi}{j}.
\end{equation}
\end{lemma} 
\begin{proof}
From Lemma \ref{scalarcurvCase1} and $R_j\geq-\frac{1}{j}$, we have
\begin{equation}\label{scalarcurvature}
R_j = \frac{a''_j}{a_j}+\frac{b''_j}{b_j}+\frac{a'_j b'_j}{a_j b_j}\leq\frac{1}{2j}
\end{equation}
Now, we compute the derivatives of $\alpha_j$ and $\beta_j$. 
\begin{equation}
\alpha'_j=\frac{a'_j}{a_j}, \quad \alpha''_j=\frac{a''_j}{a_j}-\frac{{a'_j}^2}{a_j}, \quad \beta'_j=\frac{b'_j}{b_j} \quad \text{and} \quad \beta''_j=\frac{b''_j}{b_j}-\frac{{b'_j}^2}{b_j^2}
\end{equation}
Substituting into (\ref{scalarcurvature}) above inequality we have 
\begin{equation}\label{Rbound}
\alpha''_j + \beta''_j + {\alpha'_j}^2 + {\beta'_j}^2 + \alpha'_j \beta'_j \leq\frac{1}{2j}
\end{equation}

Since $\alpha_j$ and $\beta_j$ are periodic, we may integrate this inequality to find
\begin{equation}
\int_{-\pi}^{\pi} {\alpha_j'}^2 + {\beta_j'}^2 +  \alpha_j'\beta_j'\, dz \leq \frac{\pi}{j} 
\end{equation}
\begin{equation}\label{alpha'beta'bound}\int_{-\pi}^{\pi} {\alpha_j'}^2 + {\beta_j'}^2 \,dz\leq \frac{\pi}{j} - \int_{-\pi}^{\pi}\alpha_j'\beta_j' dz
\end{equation}
Rewriting and then integrating (\ref{scalarcurvature}),
\begin{equation}
\frac{(a_j b_j)'' - a'_j b'_j}{a_j b_j}\leq\frac{1}{2j}
\end{equation}
\begin{equation}\label{intermediatebound}\int_{-\pi}^{\pi} \frac{a'_j b'_j}{a_j b_j}dz \geq \int_{-\pi}^{\pi} \frac{(a_j b_j)''}{a_j b_j}dz - \frac{\pi}{j}\end{equation}
Now, since $a_j b_j$ is periodic,
\begin{equation}
0= \int_{-\pi}^{\pi} \ln(a_j b_j)'' dz = \int_{-\pi}^{\pi} \frac{(a_j b_j)''}{a_j b_j} - \frac{{(a_j b_j)'}^2}{(a_j b_j)^2} dz
\end{equation}
Applying this identity to (\ref{intermediatebound}),
\begin{equation}\label{a'b'bound}\int_{-\pi}^{\pi} \frac{a'_j b'_j}{a_j b_j}dz \geq \int_{-\pi}^{\pi} \frac{(a_j b_j)''}{a_j b_j}dz - \frac{\pi}{j} = \int_{-\pi}^{\pi} \frac{{(a_j b_j)'}^2}{(a_j b_j)^2} dz - \frac{\pi}{j} \geq -\frac{\pi}{j}\end{equation}
Using the definition of $\alpha'_j$ and $\beta'_j$ and applying (\ref{a'b'bound}) to (\ref{alpha'beta'bound}), 
\begin{equation}
\int_{-\pi}^{\pi} {\alpha_j'}^2 + {\beta_j'}^2 dz \leq \frac{\pi}{j} - \int_{-\pi}^{\pi}\alpha_j'\beta_j' dz = \frac{\pi}{j} - \int_{-\pi}^{\pi}\frac{a'_j b'_j}{a_j b_j} dz \leq \frac{2\pi}{j}
\end{equation}
Thus, we have the desired bounds. 
\end{proof}

We now come to the most important result of this section:

\begin{proposition}\label{UpperLowerBoundsDoubly}
Let $M_j = (\mathbb{T}^3, g_j)$ as in (\ref{Case1}). If $R_j \geq -\frac{1}{j}$, $\Diam(M_j) \leq D_0$, and $\MinA(M_j) \geq A_0>0 $, then there exist positive constants $A, A', B, B'$ independent of $j$ such that 
\begin{align}
\frac{A_0}{4\pi^2D_0}e^{-\frac{2\pi}{\sqrt{j}}} &\leq a_j(z) \leq  D_0 e^{\frac{2\pi}{\sqrt{j}}}
\\\frac{A_0}{4\pi^2D_0}e^{-\frac{2\pi}{\sqrt{j}}} &\leq b_j(z) \leq  D_0 e^{\frac{2\pi}{\sqrt{j}}}
\end{align}
for all $z \in S^1$.
\end{proposition}
\begin{proof}
Using the notation of Lemma \ref{logl2bound}, we apply Cauchy-Schwarz and Lemma \ref{logl2bound}. 
\begin{equation}
\int_{-\pi}^{\pi} |\alpha'_j(z)| dz \leq \sqrt{\int_{-\pi}^{\pi} |\alpha'_j(z)|^2 dz} \sqrt{\int_{-\pi}^{\pi} dz}\leq \frac{2\pi}{\sqrt{j}}
\end{equation}
\begin{equation}\int_{-\pi}^{\pi} |\beta'_j(z)| dz \leq \sqrt{\int_{-\pi}^{\pi} |\beta'_j(z)|^2 dz} \sqrt{\int_{-\pi}^{\pi} dz}\leq \frac{2\pi}{\sqrt{j}} 
\end{equation}
So,
\begin{equation}
\begin{split}
&\ln\left(\frac{\max(a_j)}{\min(a_j)}\right)=\max(\alpha_j)-\min(\alpha_j)\leq\int_{-\pi}^{\pi} |\alpha'_j(z)| dz \leq \frac{2\pi}{\sqrt{j}} ,\\
&\ln\left(\frac{\max(b_j)}{\min(b_j)}\right)=\max(\beta_j)-\min(\beta_j)\leq\int_{-\pi}^{\pi} |\beta'_j(z)| dz \leq \frac{2\pi}{\sqrt{j}}\\
\end{split}
\end{equation}
By combining with Lemma \ref{lower},
\begin{equation}\label{maxbound}
\begin{split}
&\max(a_j)\leq e^{\frac{2\pi}{\sqrt{j}}}\min(a_j)\leq D_0 e^{\frac{2\pi}{\sqrt{j}}},\\
&\max(b_j)\leq e^{\frac{2\pi}{\sqrt{j}}}\min(b_j)\leq D_0e^{\frac{2\pi}{\sqrt{j}}}\\
\end{split}
\end{equation}
By Lemma \ref{subsec: bound from minA doubly warped}, $\min(a_j b_j)\geq  \frac{A_0}{4\pi^2}$. Then, combining \eqref{productbound} with \eqref{maxbound}, we get a uniform upper and lower bound for $a_j$ and $b_j$ as follows
\begin{equation}
\begin{split}
&\min(a_j)\geq\frac{\min(a_j b_j)}{\max(b_j)}\geq \frac{A_0}{4\pi^2D_0}e^{-\frac{2\pi}{\sqrt{j}}}, \\
&\min(b_j)\geq\frac{\min(a_j b_j)}{\max(a_j)}\geq \frac{A_0}{4\pi^2 D_0}e^{-\frac{2\pi}{\sqrt{j}}}\\
\end{split}
\end{equation}
Thus, we have the desired uniform upper and lower bounds on $a_j$ and $b_j$.
\end{proof}

\subsection{$W^{1,2}$ Convergence and Proof of the Main Result}\label{subsec: ProofOfMainThm Case 1}

In this section we are going to use the bounds on the warping functions to prove that they converge to constants in $W^{1,2}$. We then use Morrey's inequality to show this implies $C^{0, \frac{1}{2}}$ convergence. 

\begin{thm}\label{W12convergence}
Let $M_j = (\mathbb{T}^3, g_j)$ as in (\ref{Case1}). If $R_j \geq -\frac{1}{j}$, $\Diam(M_j) \leq D_0$, and $\MinA(M_j) \geq A_0>0$, then there exist nonzero constants $a_\infty, b_\infty$ such that, after possibly passing to a subsequence, $a_i \to a_\infty$, $b_i \to b_\infty$ in $W^{1,2}(S^1)$.
\end{thm} 
\begin{proof}
Using the notation of Lemma \ref{logl2bound}, we apply the Poincar\'e-Wirtinger inequality and use Lemma \ref{logl2bound} to obtain the limit as $j \to \infty$
\begin{align}
\Vert \alpha_j - \bar{\alpha}_j \Vert_2 &= \Vert \alpha_j-\frac{1}{2\pi}\int_{-\pi}^{\pi} \alpha_j dz \Vert_2\leq C\Vert \alpha'_j\Vert_2\to 0
\\ \Vert \beta_j - \bar{\beta}_j \Vert_2 &= \Vert \beta_j-\frac{1}{2\pi}\int_{-\pi}^{\pi} \beta_j dz\Vert_2 \leq C\Vert \beta'_j \Vert_2 \to 0,
\end{align}
where $C$ is a constant independent of $j$ and $\bar{\alpha}_j$ and $\bar{\beta}_j$ denote the averages of $\alpha_j$ and $\beta_j$ respectively.

From here on, we consider only the functions $\alpha_j$ as the arguments are identical for both $\alpha_j, \beta_j$. After passing to a subsequence, the above shows that we have a limiting function $\alpha_\infty$ so that
\begin{equation}
\alpha_{j_k} \to \alpha_\infty \text{ in } W^{1,2}(S^1)
\end{equation}
where $\alpha_\infty$ is a constant by the fact that
\begin{equation}
\int_{-\pi}^{\pi}|\alpha_\infty - \bar{\alpha}_{j_k} |^2 dz \leq \int_{-\pi}^{\pi}|\alpha_\infty - \alpha_{j_k} |^2 + |\alpha_{j_k} - \bar{\alpha}_{j_k} |^2 dz \to 0 \quad \text{as $j \rightarrow \infty$}.
\end{equation}
Now, by Proposition 2.4, there are positive constants $A, A'$ such that $A\leq a_j \leq A'$, thus  
\begin{equation}
\int_{-\pi}^{\pi} \alpha_j dz = \int_{-\pi}^{\pi} \ln(a_j) dz \leq  \int_{-\pi}^{\pi} \ln(A') dz = 2\pi \ln(A')
\end{equation}
and 
\begin{equation}
2\pi \ln(A) = \int_{-\pi}^{\pi} \ln(A) dz\leq \int_{-\pi}^{\pi} \alpha_j dz
\end{equation}
So, the averages $\bar{\alpha}_j$ cannot get arbitrarily large or arbitrarily small as $i\to \infty$. In particular, $\alpha_\infty$ is a positive constant. 

Now that we have found subsequences $\alpha_{j_k}$ and $\beta_{j_k}$ converging to some nonzero constants $\alpha_\infty$ and $\beta_\infty$, respectively, in $W^{1,2}(S^1)$, we can define 
$a_\infty = e^{\alpha_\infty}, b_\infty = e^{\beta_\infty}$ to obtain subsequences of $a_j, b_j$ converging to nonzero constants $a_\infty, b_\infty$ in $W^{1,2}(S^1)$.
\end{proof}

We are now ready to prove our main result for doubly warped products. 
\begin{proof}[Proof of Theorem \ref{MainThmCase1}]
By Theorem \ref{W12convergence}, we have that a subsequence of $a_j$ and $b_j$ converges in $W^{1,2}$ to constants $a_{\infty}$ and $b_{\infty}$. Applying Morrey's inequality for one-dimensional domains gives that a subsequence of $a_j$ and $b_j$ converges in $C^{0, \frac{1}{2}}$. Note that constant warping functions $a_{\infty}$, $b_{\infty}$ mean that the metric is flat. So, a subsequence of $M_j$ converges in $C^{0,\frac{1}{2}}$ to a flat torus. In particular, a subsequence  GH and SWIF converges to a flat torus by Corollary \ref{SW-SmoothConvToSWIF}.

\end{proof}

\section{Singly Warped Products of Two Variables} \label{sec: Case 2}
In this section we will prove Theorem \ref{MainThmCase2}. Recall that we are considering a sequence of singly warped product metrics $g_j$ on $\mathbb{T}^3$ such that for $x, y, z \in [- \pi, \pi]$ and positive $f_j \colon [- \pi, \pi] \times [-\pi, \pi] \rightarrow \mathbb{R}$, $g_j$ can be written as
\begin{equation}\label{Case2}
g_j = dx^2 + dy^2 + f_j^2(x, y) dz^2.
\end{equation}

The singly warped product case is substantially different than the doubly warped product case because $f_j$ is a function of two variables. This means we will not be able to apply Morrey's inequality to go from $W^{1,2}$ convergence to $C^{0,\alpha}$ convergence as we were able to do for doubly warped products.

\subsection{Scalar Curvature} \label{subsec: Scalar Calculation Case 2}
We first analyze the partial differential inequality on the warping function obtained from $R_j \geq -\frac{1}{j}$.

Applying the calculations of Dobarro and Dozo, we may find an expression for the scalar curvature of a singly warped product on $\mathbb{T}^3$ \cite{Dobarro-Dozo}. 
\begin{lemma}\label{Case2scalarcurvature}
The scalar curvature for a metric $g= dx^2 + dy^2 + f^2(x,y)dz^2$ on $\mathbb{T}^3$ is
\begin{equation}
R = - 2\frac{\Delta f}{f}
\end{equation}
where $\Delta$ is the Euclidean Laplacian. 
\end{lemma}

\begin{remark}
If we further assume that the $M_j$'s are scalar flat, i.e. $\frac{\Delta f_j}{f_j} = 0 $ then the maximum principle shows that the warping functions must be constant. This is one way to see that scalar flat $3$-tori with a singly warped product metric are isometric to a flat torus. 
\end{remark}

Lemma \ref{Case2scalarcurvature} means that the assumption on scalar curvature in Theorem \ref{MainThmCase2} translates into the following inequality for the warping functions:
\begin{equation}\label{scalarcurvatureinequality}\frac{\Delta f_j}{f_j} \leq \frac{1}{2j}\end{equation}

\subsection{Minimal Surfaces, the MinA Condition and Uniform Bounds}\label{subsec:MinA Condition Case 2} 

In this section we investigate the $\MinA$ condition in a similar fashion as in Subsection \ref{subsec: bound from minA doubly warped} in order to obtain important bounds on $f_j$ which will be used in later subsections. More precisely we will be able to prove that the $\MinA$ lower bound yields  uniform lower bounds on the simple integrals of $f_j(x_0, y)$ and $f_j(x, y_0)$, and on the double integral of $f_j(x, y)$. 

\begin{lemma}\label{bound from minA singly warped}
Let $M_j = (\mathbb{T}^3, g_j)$ as in (\ref{Case2}). If $\MinA(M_j) \geq A_0>0$, then
\begin{align}
& \int_{-\pi}^\pi \int_{-\pi}^\pi f_j(x,y)dxdy \geq A_0, \\
& \int_{-\pi}^\pi f_j(x_0, y) dy \geq \frac{A_0}{2\pi} \textit{ for all } x_0\in [-\pi, \pi],\\
&\int_{-\pi}^\pi f_j(x, y_0) dx \geq \frac{A_0}{2\pi} \textit{ for all } y_0\in [-\pi, \pi].
\end{align}
\end{lemma}
\begin{proof}
The proof is exactly as in Lemma \ref{subsec: bound from minA doubly warped}. The areas of the embeddings $\phi_{x=x_0}, \phi_{y=y_0}$ in this case are
\begin{align}
\Area(\phi_{x=x_0}) & = 2\pi \int_{-\pi}^\pi f_j(x_0,y)dy,
\end{align}
and 
\begin{align}
\Area(\phi_{y=y_0}) & = 2\pi\int_{-\pi}^\pi f_j(x,y_0)dx.
\end{align}
The first bound follows by integrating either of the bounds above.
\end{proof} 

\subsection{$W^{1,2}$ Convergence of $h_j$} \label{subsec:Convergence of hj Case 2}
Define the sequence $\{h_j\}$ by $h_j(x, y) := \ln(f_j(x, y))$, for every $j \in \mathbb{N}$. Note that these functions are defined on $\mathbb{T}^2 = [-\pi, \pi] \times [-\pi, \pi]$, since they are periodic in $x$ and $y$. Moreover, define $\bar{h}_j$ to be the average of $h_j$ over the torus $\mathbb{T}^2$, i.e. 
\begin{equation}
 \bar{h}_j = \frac{1}{\vert \mathbb{T}^2 \vert} \int_{\mathbb{T}^2} h_j \, dA
\end{equation}
where $\vert \mathbb{T}^2 \vert = 4\pi^2$ and $dA = dx dy$.
The averages $\bar{h}_j$ cannot get arbitrarily large due to the following control inequalities. 
\begin{align} \label{average h cannot get arbitrarily large}
\int_{\mathbb{T}^2} h_j dA = \int_{\mathbb{T}^2} \ln(f_j) dA \leq \ln \left(\int_{\mathbb{T}^2} f_j dA \right ) \leq \ln(\Vol(M_j)) \leq \ln(V_0).
\end{align}

We now calculate the inequality satisfied by $h_j$ 
\begin{equation}
\Delta h_j = \Delta \ln(f_j) = \frac{\Delta f_j}{f_j} - \frac{|\nabla f_j|^2}{f_j^2}.
\end{equation}
Applying (\ref{scalarcurvatureinequality}), we obtain an elliptic inequality satisfied by $h_j$
\begin{align}\label{NiceEllipticEq}
\Delta h_j +|\nabla h_j|^2 \le \frac{1}{2j}.
\end{align}
\begin{proposition}\label{proposition hj L2 convergence to average}
Let $M_j = (\mathbb{T}^3, g_j)$ as in (\ref{Case2}). Let $h_j:= \ln(f_j)$. If $R_j \geq -\frac{1}{j}$, then
$$\lim_{j\to \infty}\Vert h_j - \bar{h}_j \Vert_{L^2(\mathbb{T}^2)} = 0.$$
\end{proposition}
\begin{proof}

Since $f_j$ is periodic in both variables, $h_j$ is as well. So, $h_j$ may be thought of as a smooth function on a flat $2$-torus. Integrating \eqref{NiceEllipticEq} we find
\begin{equation} \label{IntEllipticEq}
\int_{\mathbb{T}^2}\left( \Delta h_j +|\nabla h_j|^2\right) \, dA \le \int_{\mathbb{T}^2}\frac{1}{2j} \, dA
\end{equation}
which then becomes 
\begin{equation}
\int_{\mathbb{T}^2}|\nabla h_j|^2 \, dA \le \frac{1}{2j} \vert\mathbb{T}^2\vert
\end{equation}
So, 
\begin{equation}\label{hjW12}\int_{\mathbb{T}^2}|\nabla h_j|^2 \rightarrow 0\end{equation}
as $j \rightarrow \infty$. 

 Applying the Poincar\'e-Wirtinger inequality with constant $C_{\mathbb{T}^2}$ from $\mathbb{T}^2$, we find that
\begin{equation}
\Vert h_j - \bar{h}_j \Vert_{L^2(\mathbb{T}^2)}^2 =  \int_{\mathbb{T}^2} |h_j - \bar{h}_j|^2 \, dA \leq C_{\mathbb{T}^2}^2\int_{\mathbb{T}^2}|\nabla h_j|^2 \, dA \to 0, \quad \text{as $j \rightarrow \infty$},
\end{equation}
thus establishing the claim and finishing the proof. 
\end{proof}
After \eqref{average h cannot get arbitrarily large} we are naturally inclined to investigate whether the averages $\bar{h}_j$ can get arbitrarily small as well. In order to argue that this does not happen we will show in Lemma \ref{f_j convergence} that there is subsequential $W^{1,1}$ convergence of $f_j$ to its average and moreover subsequential $W^{1,2}$ convergence of $h_j$ to its average on a subsequence with an accompanying lower bound on $\bar{h}_j$. 

\subsection{$W^{1,1}$ Convergence of $f_j$} \label{subsec: Convergence of f_j Case 2}

Our goal in this section is to show $W^{1,1}$ convergence of $f_j$. Before we proceed with that goal we will first establish a uniform $L^p$ bound, $p>2$, for the sequence.

\begin{proposition}\label{LinfinityEsthj}
    Let $M_j = (\mathbb{T}^3, g_j)$ as in (\ref{Case2}). Let $R_j \geq -\frac{1}{j}$, $\Vol(M_j)\leq V_0$, and $\MinA(M_j) \geq A_0>0$. Then for each fixed $p>2$ there exists a $C_p>0$ so that
    \begin{align}
        \int_{\mathbb{T}^2} f_j^p \le C_p.
    \end{align}
\end{proposition}
\begin{proof}
  Let $h_j=\ln(f_j)$. By the Moser-Trudinger inequality on a compact Riemann surface (see~\cite[Theorem 1.1(1)]{yang2007sharp} with $\alpha = 0$), we find
\begin{align}\label{equation Moser-Trudinger}
    \int_{[-\pi,\pi]^2} e^{\frac{(h_j-\bar{h}_j)^2}{c^2\|\nabla h_j\|^2_{L^2([-\pi,\pi]^2)}}}dxdy \le 4\pi^2.
\end{align}
Since $c\|\nabla h_j\|_{L^2([-\pi,\pi]^2)} \le c_j$ where $c_j$ is decreasing so that $c_j \rightarrow 0$ as $j \rightarrow \infty$, we can rewrite ~\eqref{equation Moser-Trudinger} as
\begin{align}
    \int_{[-\pi,\pi]^2} e^{\frac{(\ln(f_j)-\overline{\ln(f_j)})^2}{c_j^2}}dxdy = \int_{[-\pi,\pi]^2} e^{\frac{\left(\ln(f_j\bar{f}_j^{-1})\right)^2}{c_j^2}}dxdy \le 4\pi^2.
\end{align}
If we specifically look at the set $\{|\ln(f_j /\overline{f_j})|\ge 1\}$ then we find
\begin{align}
    \int_{\{|\ln\left(f_j\bar{f}_j^{-1}\right)|\ge 1\}} e^{\frac{\left|\ln\left(f_j\bar{f}_j^{-1}\right)\right|}{c_j}}dxdy \le  \int_{\{|\ln\left(f_j\bar{f}_j^{-1}\right)|\ge 1\}} e^{\frac{\left(\ln(f_j\bar{f}_j^{-1})\right)^2}{c_j^2}}dxdy \le 4\pi^2,
\end{align}
and if we further restrict to the set where $\{\ln(f_j\bar{f}_j^{-1}) \ge 1\}$ we find
\begin{align}
   \int_{\{\ln(f_j\bar{f}_j^{-1})\ge 1\}} (f_j\bar{f}_j^{-1})^{\frac{1}{c_j}}dxdy = \int_{\{\ln(f_j\bar{f}_j^{-1})\ge 1\}} e^{\frac{\ln(f_j\bar{f}_j^{-1})}{c_j}}dxdy  \le 4\pi^2,
\end{align}
which implies
\begin{align}
  \left( \int_{\{\ln(f_j\bar{f}_j^{-1})\ge 1\}} (f_j\bar{f}_j^{-1})^{\frac{1}{c_j}}dxdy \right)^{c_j}  \le (4\pi)^{2c_j}.
\end{align}
Now notice that where $-\infty<\ln(f_j\bar{f}_j^{-1})\le 1$ we know $0 < f_j\bar{f}_j^{-1} \le e$ and hence 
\begin{align}
  \left( \int_{\{-\infty<\ln(f_jf_j\bar{f}_j^{-1})\le 1\}} (f_j\bar{f}_j^{-1})^{\frac{1}{c_j}}dxdy \right)^{c_j}\le \left(\int_{[-\pi,\pi]^2} (e)^{\frac{1}{c_j}}dxdy \right)^{c_j} \le (4\pi)^{2c_j}e,
\end{align}
which implies 
\begin{align}
 \|f_j\bar{f}_j^{-1}\|_{L^{c_j^{-1}}([-\pi,\pi]^2)}= \left( \int_{[0,\pi]^2} (f_j\bar{f}_j^{-1})^{\frac{1}{c_j}}dxdy \right)^{c_j}  \le (e+1)(4\pi)^{2c_j}\le (e+1)(4\pi)^{2c_1}.
\end{align}
This is equivalent to 
\begin{align}
 \|f_j\|_{L^{c_j^{-1}}([-\pi,\pi]^2)}  \le \bar{f}_j(e+1)(4\pi)^{2c_j}\le \bar{f}_j(e+1)(4\pi)^{2c_1}.
\end{align}
  Since $ \int_{[-\pi,\pi]^2}f_jdxdy = \frac{\Vol(M_j)}{2\pi} \le \frac{V_0}{2\pi}$, $\bar{f}_j \le \frac{V_0}{8\pi^3}$, and $c_j^{-1} \rightarrow \infty$ as $j \rightarrow \infty$ the result follows.
\end{proof}

We now use this newly found control on $h_j$ and $f_j$ to find $W^{1,1}$ convergence of $f_j$ and $W^{1,2}$ convergence of $h_j$.

\begin{lemma}\label{f_j convergence} Let $M_j = (\mathbb{T}^3, g_j)$ as in (\ref{Case2}). Let $R_j \geq -\frac{1}{j}$, $\Vol(M_j)\leq V_0$, and $\MinA(M_j) \geq A_0>0$. Then, for some constant $f_{\infty} \in (0,\infty)$ and some subsequence $f_{j_k}$, $f_{j_k} \to f_{\infty}\in (0,\infty)$ in $W^{1,1}$. Similarly, if $h_j := ln(f_j)$, then for some subsequence and some constant $h_{\infty} \in \R$, $h_{j_k} \to h_{\infty}$ in $W^{1,2}$.
\end{lemma}
\begin{proof}
Let $h_j := \ln(f_j)$. By (\ref{hjW12}),
\begin{equation}\label{equation W12 for h}
\int_{\mathbb{T}^2}\frac{|\nabla f_j|^2}{f_j^2}\, dA = \int_{\mathbb{T}^2}|\nabla h_j|^2 \, dA\to 0, \quad \text{as $j \rightarrow \infty$}
\end{equation}
Now we calculate
\begin{align}
\int_{\mathbb{T}^2}|\nabla f_j| \, dA&=\int_{\mathbb{T}^2}\frac{|\nabla f_j|}{f_j} f_j \, dA 
\\&\leq \left(\int_{\mathbb{T}^2}\frac{|\nabla f_j|^2}{f_j^2} \, dA \right)^{1/2} \left (\int_{\mathbb{T}^2}f_j^2dA\right)^{1/2}
\\&\le C_0 \left (\int_{\mathbb{T}^2}\frac{|\nabla f_j|^2}{f_j^2}\, dA \right)\rightarrow 0
\end{align}
Where the $L^2$ upper bound on $f_j$ comes from Proposition \ref{LinfinityEsthj}.

Then by Lemma \ref{bound from minA singly warped} combined with the the uniform bound on $\|f_j\|_{L^1} = \Vol(M_j) \leq V_0$, we have that $\bar{f}_j = \frac{1}{|\mathbb{T}^2|}\int_{\mathbb{T}^2}f_jdA$ is uniformly bounded above and below by positive constants and so some subsequence $\bar{f}_{j_k}$ converges to a constant $\bar{f}_{\infty}$. Then, by using the Poincar\'e inequality we find
 \begin{equation}
 \int_{\mathbb{T}^2}|\nabla f_{j_k}|\, dA\ge  \int_{\mathbb{T}^2}|f_{j_k}-\bar{f}_{j_k}| \, dA 
 \end{equation}
which gives the convergence of $f_{j_k} \to \bar{f}_{\infty}\in (0,\infty)$ in $L^1$. Since $\nabla \bar{f}_{\infty} \equiv 0$, we in fact have that $f_{j_k} \to \bar{f}_{\infty}$ in $W^{1,1}$. After relabelling $\bar{f}_{\infty}$ by $f_{\infty}$, we conclude the first part of this lemma.

Similarly, we find that $h_{j_k} \to h_{\infty} \in \R$ in $W^{1,2}$. Since $f_{j_k} \to \bar{f}_{\infty}\in (0,\infty)$ in $L^1$, we can choose a further subsequence so that $f_{j_k} \to \bar{f}_{\infty}\in (0,\infty)$ pointwise almost everywhere. Thus, $h_{j_k} \rightarrow h_{\infty} :=\ln(\bar{f}_{\infty})$ pointwise almost everywhere. 

By Egorov's theorem, for each $\epsilon>0$, there is a measurable subset $A_{\epsilon} \subseteq \mathbb{T}^2$ such that $\Vol(A_{\epsilon})< \epsilon$ and $h_{j_k}$ converges uniformly on $\mathbb{T}^2\setminus A_{\epsilon}$. In particular, there exists $k_0>0$ such that for all $k > k_0$, $h_{j_k}>h_{\infty}-1$ everywhere on $\mathbb{T}^2\setminus A_{\epsilon}$. 

First, we note that $\sup_k \bar{h}_{j_k}< \infty$ by~\eqref{average h cannot get arbitrarily large}. Now, we will show that $\liminf_{k \to \infty} \bar{h}_{j_k}>-\infty$. Suppose not. That is, suppose that there is a further subsequence $h_{j_k}$ such that $\lim_{k\to\infty} \bar{h}_{j_k} = -\infty$. Then, in particular, there would exist $k_1 > 0$ such that $\bar{h}_{j_k} < h_{\infty}-2$ for all $k>k_1$. So, for $k>\max(k_0, k_1)$,
\begin{align}
    \int_{\mathbb{T}^2} |h_{j_k} - \bar{h}_{j_k}|^2\, dA &= \int_{\mathbb{T}^2\setminus A_{\epsilon}} |h_{j_k} - \bar{h}_{j_k}|^2\, dA +\int_{A_{\epsilon}} |h_{j_k} - \bar{h}_{j_k}|^2\, dA \\
    &\geq \int_{\mathbb{T}^2\setminus A_{\epsilon}}|h_{j_k} - \bar{h}_{j_k}|^2\, dA\\ \label{HalfWayEq} 
\end{align}
Since $k> \max(k_0, k_1)$, we have both that $h_{j_k}>h_{\infty}-1$ on $\mathbb{T}^2\setminus A_{\epsilon}$ and that $\bar{h}_{j_k} < h_{\infty}-2$, which implies that $|h_{j_k} - \bar{h}_{j_k}|>1$ on $\mathbb{T}^2\setminus A_{\epsilon}$. Hence, by continuing the estimate in \eqref{HalfWayEq} we find
\begin{align}
  \int_{\mathbb{T}^2} |h_{j_k} - \bar{h}_{j_k}|^2\, dA  &\geq \int_{\mathbb{T}^2\setminus A_{\epsilon}} 1\, dA\\
    &= \Vol(\mathbb{T}^2) - \Vol(A_{\epsilon})\\
    &= \Vol(\mathbb{T}^2) - \epsilon
\end{align}
Since $\epsilon$ was arbitrary, we find a contradiction with Proposition \ref{proposition hj L2 convergence to average}. Thus, 
$$\liminf_{k \to \infty} \bar{h}_{j_k}>-\infty.$$
Thus, we find that there exists a $C>0$ so that $-\infty < -C\le\bar{h}_{j_k}\le C< \infty$. Combined with Proposition \ref{proposition hj L2 convergence to average}, we have that $h_{j_k} \to h_{\infty} \in \mathbb{R}$ in $L^2$ along a subsequence. Combined with~\eqref{equation W12 for h}, we find that $h_{j}$ subsequentially converges to $h_{\infty} \in \mathbb{R}$ in $W^{1,2}$.
\end{proof}

\subsection{$C^0$ Convergence from Below}\label{subsec:C0 From Below Case 2}
Now, we have from Lemma \ref{f_j convergence} that on some subsequence, $f_j$ converges in $W^{1,1}$ to a positive constant. We would like to use this to show convergence of $M_j$, as in (\ref{Case2}), to a flat torus. It was shown in \cite{BALp} that if the warping functions are bounded in $L^p$ for $p>2$ (or equivalently the metrics bounded in $L^\frac{p}{2}$), volume converges, and the distance function is $C^0$ converging from below, then the sequence is converging in the uniform, GH, and SWIF sense. We will first show the $C^0$ convergence from below by using a maximum principle argument on the operator $L f = \Delta f + |\nabla f|^2$. By the inequality in equation \eqref{NiceEllipticEq} we expect to be able to bound the minimum of $h_j$ using the maximum principle as we now proceed to do.  
\begin{lemma}\label{minControl}
Let $M_j = (\mathbb{T}^3, g_j)$ as in (\ref{Case2}). Let $R_j \geq -\frac{1}{j}$. Let $h_j := \ln(f_j)$. Then, for $\Omega=[\eta_1,\eta_2]\times S^1\subset \mathbb{T}^2 = -[\pi,\pi]\times [-\pi,\pi]$, we have
\begin{equation}
\min_{\Omega} h_j \ge \min_{\partial \Omega} h_j - (e^{\gamma_j \eta_2} - e^{\gamma_j \eta_1})
\end{equation}
where $\gamma_j=\sqrt{\frac{C}{2j}}$.
\end{lemma}
\begin{proof}
Consider the function $h_j -e^{\gamma_j \theta_1}$, $\theta_1 \in [\eta_1,\eta_2]$, $\gamma_j > 0$, and compute
\begin{align}
L(h_j -e^{\gamma_j \theta_1}) &= \Delta(h_j -e^{\gamma_j \theta_1}) + |\nabla( h_j-e^{\gamma_j \theta_1})|^2
\\&=\Delta(h_j -e^{\gamma_j \theta_1}) + |\nabla h_j|^2 -2\langle \nabla h_j,\nabla e^{\gamma_j \theta_1}\rangle +|\nabla e^{\gamma_j \theta_1}|^2
\\&=L(h_j) -2\langle \nabla( h_j-e^{\gamma_j \theta_1}),\nabla e^{\gamma_j \theta_1}\rangle - |\nabla e^{\gamma_j \theta_1}|^2 -\Delta e^{\gamma_j \theta_1}.
\end{align}
Thus, we obtain the identity 
\begin{align}
L(h_j -e^{\gamma_j \theta_1})&+2\langle \nabla( h_j-e^{\gamma_j \theta_1}),\nabla e^{\gamma_j \theta_1}\rangle =L(h_j)  - |\nabla e^{\gamma_j \theta_1}|^2 -\Delta e^{\gamma_j \theta_1},
\end{align}
whose right-hand side can be bounded as follows, using (\ref{NiceEllipticEq}), 
\begin{align}
L(h_j)  - |\nabla e^{\gamma_j \theta_1}|^2 -\Delta e^{\gamma_j \theta_1} \le \frac{1}{2j} -\gamma_j^2(e^{2\gamma_j \theta_1}+e^{\gamma_j \theta_1}) \le \frac{1}{2j} -\gamma_j^2 C' \leq 0,
\end{align}
where we uniformly bound the exponential terms independent of $j$ and choose $\gamma_j = \sqrt{\frac{C}{2j}}$ for some $C$ independent of $j$ so that the last inequality holds. Then, by the minimum principle, we know that the minimum must be obtained on the boundary, i.e.
\begin{align}
\min_{\Omega} h_j - e^{\gamma_j \eta_1} \ge \min_{\Omega} \left(h_j- e^{\gamma_j \theta_1}\right) \ge \min_{\partial \Omega} \left(h_j- e^{\gamma_j \theta_1}\right)\ge \min_{\partial \Omega} h_j - e^{\gamma_j \eta_2}.
\end{align}
\end{proof} 

Now in order to effectively use Lemma \ref{minControl} we must be able to control $h_j$ on $\partial \Omega$ and so now we obtain this control for a subsequence.
\begin{lemma}\label{C0S1}
If $h_j \rightarrow h_{\infty}$ in $W^{1,2}(\mathbb{T}^2)$ and if $h^{\bar{y}}_j(x) := h_j(x,\bar{y})$ for $\bar{y} \in [-\pi,\pi]$, then for some subsequence, $h^{\bar{y}}_{j_k}(x) \rightarrow h_{\infty}$ in $C^0([-\pi,\pi])$, for almost every $\bar{y} \in [-\pi,\pi]$.
\end{lemma}
\begin{proof}
Since for some subsequence, $h_{j_k} \rightarrow h_{\infty}$ in $W^{1,2}(\mathbb{T}^2)$, we know that
\begin{equation}
\int_{-\pi}^{\pi}\left(\int_{-\pi}^{\pi}|h_j-h_{\infty}|^2 +\left|\frac{\partial h_j}{\partial x}\right|^2+\left|\frac{\partial h_j}{\partial y}\right|^2 dx \right )dy \longrightarrow 0, 
\end{equation}
as $j \rightarrow \infty$, but this implies that 
\begin{equation}
\int_{-\pi}^{\pi}|h^{\bar{y}}_{j_k} - h_{\infty}|^2 + \left|\frac{\partial h^{\bar{y}}_{j_k}}{\partial x}\right|^2 dx  \longrightarrow 0
\end{equation}
for a.e. $\bar{y} \in [-\pi,\pi]$, as $k \rightarrow \infty$. This means that $h^{\bar{y}}_{j_k} \rightarrow h_{\infty}$ in $W^{1,2}([-\pi,\pi])$ and so, by Morrey's inequality, we find that $h^{\bar{y}}_{j_k} \rightarrow h_{\infty}$ in $C^0$, for almost every $\bar{y} \in [-\pi,\pi]$, as desired.
\end{proof}

By combining Lemma \ref{minControl} with Lemma \ref{C0S1} we obtain $C^0$ control from below.

\begin{corollary}\label{C0Boundh}
Let $M_j = (\mathbb{T}^3, g_j)$ as in (\ref{Case2}). Let $R_j \geq -\frac{1}{j}$, $\Vol(M_j)\leq V_0$, and $\MinA(M_j) \geq A_0>0$. Let $h_j := \ln(f_j)$. Then, after passing to a subsequence, we have the inequality
\begin{align} \label{bound h}
h_{j_k} \ge h_{\infty} - \frac{C}{k}
\end{align}
on $\mathbb{T}^2$, from which we deduce 
\begin{align}
f_{j_k} \ge f_{\infty} - \frac{\bar{C}}{k},
\end{align}
again on $\mathbb{T}^2$.
\end{corollary}
\begin{proof}
We may apply Lemma \ref{f_j convergence}, which allows us to apply Lemma \ref{C0S1}. So, we know that if we define $h^{\bar{y}}_j(x) = h_j(x,\bar{y})$, for $\bar{y} \in [-\pi,\pi]$, we find that $h^{\bar{y}}_{j_k}(x) \rightarrow h_{\infty}$ in $C^0([-\pi,\pi])$, for almost every $\bar{y} \in [-\pi,\pi]$. We can pick a $\eta_1,\eta_2 \in [-\pi,\pi]$ so that we get the desired $C^0$ convergence on $S^1 \times \{\eta_1\}$ and $S^1\times \{\eta_2\}$. Now we can apply Lemma \ref{minControl} on $S^1\times [\eta_1,\eta_2]$ and $S^1\times [\eta_2,\eta_1+2\pi]$ in order to achieve the desired bound (\ref{bound h}). Exponentiating both sides of (\ref{bound h}),
\begin{align}
f_k &\ge e^{\ln(f_{\infty}) - \frac{C}{k}}=
f_{\infty}e^{\frac{-C}{k}}, 
\end{align}
gives the desired bound for $f$. 
\end{proof}

\subsection{SWIF Convergence to a Flat Tori}\label{subsec:SWIF Conv Case 2}

We are now able to conclude with the proof of our main theorem. 
\begin{proof}[Proof of Theorem \ref{MainThmCase2}]
The $C^0$-bound from below given in Corollary \ref{C0Boundh} combined with the $L^p$, $p > 2$ bound of Proposition \ref{LinfinityEsthj} and the $W^{1, 1}$-convergence of Lemma \ref{f_j convergence} allows us to apply Theorem 1.4 of \cite{BALp} to obtain uniform, GH, and SWIF convergence to a flat torus on a subsequence. Note that an $L^p$ bound on the warping factor $f_j$ is equivalent to an $L^{\frac{p}{2}}$ on $g_j$ for warped products, which is the required hypothesis for Theorem 1.4 of \cite{BALp}.
\end{proof}

\bibliography{WarpedTori}
\bibliographystyle{plain}

\end{document}